\documentclass[reqno]{amsart}

\usepackage{color} 
\usepackage{amssymb}
\usepackage{bbold}
\usepackage{paralist} 					

\usepackage[all]{xy} 						
\usepackage{tikz}\usetikzlibrary{arrows,calc,patterns,decorations.pathreplacing}

\renewcommand{\phi}{\varphi}

\newcommand{\dcup}{\ensuremath{\mathaccent\cdot\cup}}

\newcommand{\FF}{{\mathcal F}}

\newcommand{\NN}{\mathbb{N}}

\newcommand{\RR}{\mathbb{R}}
\newcommand{\CC}{\mathcal{C}}

\newcommand{\ZZ}{\mathbb{Z}}

\newcommand{\kommentar}[1]{}

\newcommand{\charFunc}{\mathbb{1}}

\renewcommand{\epsilon}{\varepsilon}

\newcommand{\polygon}[2]{%
  let \n{len} = {2*#2*tan(360/(2*#1))} in
 ++(0,-#2) ++(\n{len}/2,0) \foreach \x in {1,...,#1} { -- ++(\x*360/#1:\n{len})}}

\definecolor{light-gray}{gray}{0.60}

\newtheorem{theorem}{Theorem}[section]
\newtheorem{lemma}[theorem]{Lemma}

\theoremstyle{definition}
\newtheorem{defi}[theorem]{Definition}

\newtheorem{corollary}[theorem]{Corollary}

\theoremstyle{remark}
\newtheorem{remark}[theorem]{Remark}

\newcommand{\be}{\begin{equation}}
\newcommand{\ben}{\end{equation}}

\numberwithin{equation}{section}

\newcommand{\hide}[1]{}

\begin{document}

\title{Wavelet Riesz Bases Associated to Nonisotropic Dilations}

\author{Hartmut F\"{u}hr}
\address{Lehrstuhl A f\"ur Mathematik, RWTH Aachen, D-52056 Aachen, Germany}
\email{fuehr@matha.rwth-aachen.de}

\author{Yannic Maus$^*$}
\address{Lehrstuhl f\"ur Algorithmen \& Komplexit\"at, Universit\"at Freiburg, D-79110 Freiburg, Germany }
\email{yannic.maus@cs.uni-freiburg.de}



\date{\today}


\begin{abstract}
A bounded, Riemann integrable and  measurable set $K\subset \RR^d$, which fulfills
\[\sum\limits_{\gamma\in\Gamma}\charFunc_K(x-\gamma)=k\text{ almost everywhere, $x\in\RR^d$}\] 
for a lattice $\Gamma\subset\RR^d$ is called $k$-tiling.
If $K\subset\RR^d$  is $k$-tiling $L^2(K)$ will admit a Riesz basis of exponentials. We use this result to construct generalized Riesz wavelet bases of $L^2(\RR^2)$, arising from the action of suitable subsets of the affine group. One example of our construction is the first known shearlet Riesz basis. 
\end{abstract}

\maketitle

\section{Introduction \& Preliminaries}
The purpose of this note is the construction of certain Riesz bases of ${\rm L}^2(\mathbb{R}^d)$ arising as dilations and translations of a single mother wavelet $\psi$. Here we consider the case where $\psi$ arises from a {\bf wavelet set}, i.e., up to normalization, the Fourier transform of $\psi$ is the characteristic function of a Borel set \cite{FaWa,SPELARDAI97,Wa}. The case of orthonormal wavelet sets is fairly well understood. They can be characterized by a joint tiling property; see Theorem \ref{thm:main_old} below, and the subsequent remark. Our main general result, which yields sufficient conditions on a set of generators to form a wavelet Riesz basis, can best be appreciated by contrasting it to the criterion from Theorem \ref{thm:main_old}. For a more concrete discussion, let us now introduce some notation.

Given a Hilbert space $\mathcal{H}$, a {\bf Riesz basis} of $\mathcal{H}$ is a family of vectors $(g_i)_{i \in I}$ in $\mathcal{H}$ spanning a dense subspace, and fulfilling the inequalities 
\[
 C_1 \sum_{i \in I} |c_i|^2 \le \left\| \sum_{i \in I} c_i g_i \right\|^2 \le C_2  \sum_{i \in I} |c_i|^2
\] for suitable constants $0< C_1 \le C_2$, and all finitely supported families $(c_i)_{i \in I}$ of complex coefficients. $C_1$ and $C_2$ are called {\bf lower- and upper Riesz basis bounds}. Given such a Riesz basis, there exists a unique {\bf dual Riesz basis} $(\widetilde{g}_i)_{i \in I}$ fulfilling $\langle g_i, \widetilde{g}_j \rangle = \delta_{i,j}$. As a consequence, every $f \in \mathcal{H}$ has the unconditionally converging expansions
\[
f = \sum_{i \in I} \langle f, g_i \rangle \widetilde{g}_i = \sum_{i \in I} \langle f, \widetilde{g}_i \rangle g_i, 
\] each with unique coefficients. Finally, given any coefficient sequence $(c_i)_{i \in I} \in \ell^2(I)$,  the sum
\begin{equation} \label{eqn:interpol}
f = \sum_{i \in I} c_i \widetilde{g}_i 
\end{equation}  converges unconditionally in the norm, and yields an $f \in \mathcal{H}$ with $\langle f, g_i \rangle = c_i$. 
For these and other facts concerning Riesz bases, we refer to \cite{CHR02}. 

We are interested in Riesz bases obtained by translating and dilating a suitably chosen function $\psi$, the wavelet. 
For $\lambda\in\RR^d$ and $A\in Gl(d,\RR)$ the translation, modulation and dilation operator are defined as follows
\begin{align}
T_{\lambda}:& L^2(\RR^d)\rightarrow L^2(\RR^d), & (T_{\lambda}f)(x)&=f(x-\lambda),\\
E_{\lambda}: &L^2(\RR^d)\rightarrow L^2(\RR^d), & (E_{\lambda}f)(x)&=e^{2\pi i \langle\lambda, x\rangle}f(x),\\
D_A:& L^2(\RR^d)\rightarrow L^2(\RR^d), & (D_Af)(x) &= |\det(A)|^{-1/2}f(A^{-1}x).
\end{align}
A {\bf wavelet orthonormal basis} is an orthonormal basis of $L^2(\mathbb{R}^d)$ of the type
\begin{equation}
\label{def:ONBwavelet}
 \left( D_A T_\lambda \psi \right)_{A \in \mathcal{D}, \lambda \in \Lambda},
\end{equation} with suitable choices of dilation matrices $\mathcal{D} \subset Gl(\mathbb{R}^d)$, translations $\Lambda \subset \mathbb{R}^d$, and of the mother wavelet $\psi \in L^2(\mathbb{R}^d)$. Any such $\psi$ is called a {\boldmath$(\Lambda,\mathcal{D})$}\textbf{-wavelet}. If the system is a Riesz basis, $\psi$ will be called a  {\boldmath$(\Lambda,\mathcal{D})$}\textbf{-Riesz-wavelet}; any $(\Lambda,\mathcal{D})$-wavelet is a $(\Lambda,\mathcal{D})$-Riesz-wavelet.

Following \cite{FaWa,DAI98,SPELARDAI97,Wa}, we call a Borel set $W \subset \mathbb{R}^d$ a {\bf {\boldmath$(\Lambda,\mathcal{D})$}-(Riesz-)wavelet set} if $\psi$ with $\widehat{\psi} = c_\psi \mathbb{1}_W$ is a $(\Lambda,\mathcal{D})$-(Riesz-)wavelet. 


The essential notion for understanding the wavelet set property is that of a {\em tiling set} w.r.t. a (full-rank) lattice. 
A full rank lattice is any set $\Gamma\subset \RR^d$ which can be written as $L\ZZ^d\subset \RR^d$, where $L\in Gl(d,\RR)$. 
Given a set $W \subset \mathbb{R}^d$ and $A\in Gl(d,\RR)$, we denote $AW=\{Ax\mid x\in W\}$ and say that $W$ is 
\begin{itemize}
 \item {\em translationally tiling} w.r.t. the lattice $\Gamma \subset \mathbb{R}^d$,  if 
 \begin{equation}
  \sum_{\gamma \in \Gamma} \mathbb{1}_{\gamma +W} (\xi) = 1,\label{def:kTiling}
	\end{equation}
  for almost every $\xi \in \mathbb{R}^d$. 
 \item  {\em multiplicatively tiling} w.r.t the (usually countable) set $\mathcal{D} \subset Gl(d,\mathbb{R})$, if 
 \[
  \sum_{A \in \mathcal{D}} \mathbb{1}_{AW} (\xi) = 1, 
 \] for almost every $\xi \in \mathbb{R}^d$. 
\end{itemize}


Before we formulate the fundamental result concerning wavelet sets and simultaneous tilings, we need to introduce the notion of a dual lattice: Given a lattice $\Gamma=L\ZZ^d \subset \mathbb{R}^d, L\in Gl(d,\RR)$, its dual is given by \mbox{$\Gamma^\bot = \{ \xi \in \mathbb{R}^d: \forall \gamma\in\Gamma~ \langle \xi, \gamma \rangle \in \mathbb{Z} \}$}. $\Gamma^{\bot}$ is equal to $L^{-T}\ZZ^d$ and is a lattice as well.
Furthermore, we let $A^{-T}$ denote the transpose inverse of the invertible matrix $A$, and we let \mbox{$\mathcal{D}^{-T} = \{ A^{-T} : A \in \mathcal{D} \}$}, for any subset $\mathcal{D} \subset Gl(d,\mathbb{R})$.

The following theorem is a special case of  \cite[Theorem 1.1]{Wa}. 
\begin{theorem}[\cite{Wa}] \label{thm:main_old}
 Let $\mathcal{D} \subset Gl(d,\mathbb{R})$, and $\Gamma \subset \mathbb{R}^d$ denote a lattice. $W \subset \mathbb{R}^d$ is a $(\Gamma^\bot,\mathcal{D})$ wavelet set iff it is translationally tiling w.r.t. $\Gamma$ as well as multiplicatively tiling w.r.t. $\mathcal{D}^{-T}$.
\end{theorem}

Below, in Theorem \ref{thm:main}, we will formulate and prove an extension of the sufficiency part of Theorem \ref{thm:main_old}. Essentially, we will trade in the ONB requirement in the definition of a wavelet set for the less restrictive Riesz basis property, and in turn obtain less prohibitive restrictions on $W$.
 The precise formulation uses a refinement of the tiling property:
\begin{itemize}
 \item A set $W \subset\mathbb{R}^d$ is {\em translationally $k$-tiling} w.r.t. the lattice $\Gamma \subset \mathbb{R}^d$,  if 
 \[
  \sum_{\gamma \in \Gamma} \mathbb{1}_{\gamma + W} (\xi) = k,
 \] for almost every $\xi \in \mathbb{R}^d$.
 \end{itemize}
 We then have:
\begin{theorem} \label{thm:main}
  Let $\mathcal{D} \subset Gl(d,\mathbb{R})$, and $\Gamma \subset \mathbb{R}^d$ denote a lattice. If $W \subset \mathbb{R}^d$ is translationally $k$-tiling w.r.t. $\Gamma$ as well as multiplicatively tiling w.r.t. $\mathcal{D}^{-T}$, then there exists $\Lambda \subset \mathbb{R}^d$ such that $W$ is a $(\Lambda,\mathcal{D})$-Riesz-wavelet set.
\end{theorem}
Note that in contrast to Theorem \ref{thm:main_old} the set $\Lambda$ in Theorem \ref{thm:main} is not a lattice; its proof reveals that it is a so called quasicrystal (consult \cite{MatMey10} for the definition of a quasicrystal). 

The purpose of this paper is twofold: In Section \ref{sec:proofThm}, we will present a proof of \mbox{Theorem \ref{thm:main}}, using a recent result on Riesz bases of exponentials \cite{GrepstadLev12,Kol_15} that relies on the sampling results for quasicrystals by Matei and Meyer, \cite{GrepstadLev12, MatMey10}. We will then present several applications of Theorem \ref{thm:main}, and construct Riesz bases in settings in which, until now, only frames were known to exist. In particular, we will exhibit a shearlet Riesz basis. 

\section{Proof of Theorem \ref{thm:main}}
\label{sec:proofThm}
For a countable $\Lambda\subset\RR^d$ we define the sequence of corresponding exponential functions by
\[E(\Lambda)=\left\{\left. x\mapsto e^{2\pi i \langle\lambda, x\rangle}\right| \lambda\in \Lambda\right\}.\] Let $\FF:L^2(\RR^d)\rightarrow L^2(\RR^d),~f\mapsto \widehat{f}$
 denote the Fourier transform which is defined pointwise by $\widehat{f}(\zeta)=\int_{\RR^d}f(x)e^{-2\pi i \langle\zeta, x\rangle}dx$ for \mbox{$f\in L^1(\RR^d)\cap L^2(\RR^d)$}. 
By a Riemann integrable set we mean a measurable set whose boundary has measure zero.

The following reasoning is a fairly straightforward adaptation of the argument developed for orthonormal wavelet sets, see e.g. \cite[page 40]{DAI98}: We will repeatedly use the simple observation that the unitary image of a Riesz basis is a Riesz basis as well, with the same bounds. Hence, using the unitarity of the Fourier transform, and the well-known relations between translation, modulation and dilation operators in connection with the Fourier transform, one finds that the wavelet system 
$\left( D_{A} T_\lambda \psi \right)_{A \in \mathcal{D}, \lambda \in \Lambda}$ is a Riesz basis of $L^2(\mathbb{R}^d)$ iff its Fourier transform image, given by 
\begin{equation}\label{eqn:FourierImage}
 \left( D_{A^{-T}} E_\lambda \widehat{\psi} \right)_{A \in \mathcal{D}, \lambda \in \Lambda} 
\end{equation} is a Riesz basis of $L^2(\RR^d)$. For $\widehat{\psi} = \mathbb{1}_W$, with a multiplicatively tiling set $W$, this problem can be reduced to that of constructing a suitable Riesz basis of exponentials on $L^2(W)$, by the following simple observation.
\begin{lemma}
\index{Dilation Lemma}
\label{dilationBasis}
Let $W\subset\RR^d$ be a measurable and bounded set, $A\in Gl(d,\RR)$ and $(f_i)_{i\in I}$ a Riesz basis of $L^2(W)$ with Riesz bounds $C_1, C_2>0$.
Then $(D_Af_i)_{i\in I}$ is a Riesz basis of $L^2(AW)$ with Riesz bounds $C_1$ and $C_2$.

If $W$ is multiplicatively tiling with respect to $\mathcal{D}'\subset Gl(d,\RR)$,  then $(D_Af_i)_{i\in I,A\in \mathcal{D}'}$ is a Riesz basis of $L^2(\RR^d)$.
\end{lemma}
\begin{proof}
Since $D_A: L^2(W) \to L^2(AW)$ is unitary, the first statement follows from the comment on unitary images of Riesz bases made above. 
For the proof of the second statement, first note that the tiling property of $W$ yields 
\[
 L^2(\mathbb{R}^d) = \bigoplus_{A \in \mathcal{D}'} L^2(AW).
\] Now, it is easily seen that a union of Riesz bases of $L^2(AW)$ with identical bounds $C_1,C_2$, for each $A \in \mathcal{D'}$, is a Riesz basis of the orthogonal sum, with the same Riesz basis bounds. 
%
%
%
\end{proof}

With Lemma \ref{dilationBasis} and the preceeding comments, we are left with the ``local'' problem of constructing a Riesz basis of exponentials for $L^2(W)$ to complete the proof of Theorem \ref{thm:main}. This is solved by the following result due to Grepstad and Lev, recently somewhat generalized by Kolountzakis. Its proof relies on sampling results for quasicrystals by Matei and Meyer \cite{MatMey10}.
\begin{theorem}[\cite{GrepstadLev12,Kol_15}]
\label{GrepstadTheorem}
Let $W$ be a bounded, Riemann integrable set which is $k$-tiling with respect to the lattice $\Gamma$. Then there is a set $\Lambda\subset\RR^d$ such that $E(\Lambda)$ is a Riesz basis of exponentials of $L^2(W)$.
\end{theorem}
The proof of Theorem \ref{GrepstadTheorem}  does not use the structure of $W$ except for its measurability and integrability. Especially the constructed Riesz basis of exponentials only depends on the tiling lattice $\Gamma$ and the tiling constant $k$ but not on the set $W$ itself. If $W_1$ and $W_2$ are both $k$-tiling with the same lattice and the same $k$, Theorem \ref{GrepstadTheorem} will result in a sequence of exponentials such that its restriction to $W_i$ provides a Riesz basis of $L^2(W_i)$ for $i=1,2$.

\subsubsection*{Proof of Theorem \ref{thm:main}} Let $W$ be translationally $k$-tiling w.r.t. $\Gamma$ as well as multiplicatively tiling w.r.t. $\mathcal{D}^{-T}$. Then Theorem \ref{GrepstadTheorem} provides a Riesz basis of exponentials of $L^2(W)$. Lemma \ref{dilationBasis} applied with this Riesz basis and $\mathcal{D}' = \mathcal{D}^{-T}$ as well as the comments preceding the lemma conclude the proof of the theorem.

\section{Riesz basis discretization of continuous wavelet transforms}
In section \ref{sec:examples}, we apply the previous results to special instances of dilation sets $\mathcal{D}$, where the resulting Riesz bases can be understood as yielding a discretization of a suitable continuous wavelet transform. In this section we explain this discretization formally in \mbox{Theorem \ref{thm:cwt_basis}}, a special case of Theorem \ref{thm:main}. 

 We shortly introduce the relevant terminology: We consider a suitable closed subgroup $H \subset Gl(d,\mathbb{R})$ 
 and the associated semidirect product group $G=\mathbb{R}^d \rtimes H$, which is the affine group generated by $H$ and all translations $\RR^d$. Elements of this group are denoted by pairs $(x, A) \in \mathbb{R}^d \times H$. The associated continuous wavelet transform fixes a suitable $\psi \in L^2(\mathbb{R}^d)$ and maps each $f \in L^2(\mathbb{R}^d)$ onto the function
\[
 \mathcal{W}_\psi f: \mathbb{R}^d \rtimes H \to \mathbb{C}~, (x,A) \mapsto \langle f, T_x D_A \psi \rangle;\]
 note the different order of operators, in comparison to the definition of wavelet bases in (\ref{def:ONBwavelet}). 
Under suitable assumptions on $H$ and $\psi$, which will be fulfilled in both examples in Section \ref{sec:examples}, the continuous wavelet transform is isometric up to a constant, i.e, one has
\begin{equation} \label{eqn:cwt_isometric}
 \| f \|_2^2 = \frac{1}{c_\psi} \int_{H} \int_{\mathbb{R}^d} |\mathcal{W}_\psi f(x,A)|^2 dx \frac{dA}{|{\rm det}(A)|}.
\end{equation}
Here $dA$ denotes integration with respect to the left Haar measure on $H$.

For the rest of the paper, we consider the case where $H$ has  finitely many {\bf open, free dual orbits}.
I.e., we assume the existence of finitely many open subsets $\mathcal{O}_i \subset \mathbb{R}^d$ such that, given any choice of $\xi_i \in \mathcal{O}_i$, for $i=1,\ldots, m$, the induced map 
\[
 \Phi_i: H\rightarrow \mathcal{O}_i\subset \RR^d,  A \mapsto A^{-T} \xi_i
\] is a bijection onto $\mathcal{O}_i$. Here we assume that the $\mathcal{O}_i$ are different, and thus necessarily pairwise disjoint. This existence of open free orbits was recognized in \cite{BeTa} as being important for the construction of continuous wavelet transforms in higher dimensions (the freeness condition may be somewhat relaxed to compactness of the associated stabilizers \cite{Fu10}), and it is fulfilled by a large variety of matrix groups, see e.g. \cite{Mu,BeTa,Fu_abelian,DAHLKE10} for various examples of such groups. 

In order to construct wavelet Riesz bases in this setting, we furthermore assume the existence of a {\bf quasi-lattice} $\mathcal{D} \subset H$: I.e., $\mathcal{D}$ is a discrete subset of $H$ with the following property:
\begin{itemize}
\item There exists a relatively compact open set $\CC \subset H$ with boundary of measure zero such that 
\begin{equation}\label{def:quasiLattice}
 \sum_{A \in \mathcal{D}} \mathbb{1}_{A\CC} = 1
\end{equation} holds almost everywhere (with respect to the left Haar measure on $H$). We then call $\CC$ a {\bf complement} to $\mathcal{D}$.
\end{itemize}
The set $\mathcal{D}$ can be understood as an (essentially uniform) discretization of $H$.

With these definitions in place, we can now formulate an existence result for Riesz bases  which can be understood as a discretization of the  continuous wavelet transform; see Remark \ref{rem:sampling} for a more detailed explanation. 

\begin{theorem} \label{thm:cwt_basis}
Let $H$ be a matrix group with open free dual orbits 
 \[ \mathcal{O}_i = \{ A^{-T} \xi_i ~:~ A \in H \} \]
 for suitable $\xi_1,\ldots,\xi_m \in \mathbb{R}^d$ such that $\mathbb{R}^d \setminus \bigcup_{i=1}^m \mathcal{O}_i$ has measure zero, and let  $\mathcal{D} \subset H$ be a quasi-lattice with complement $\CC \subset H$. If there exists a lattice $\Gamma$ and a natural number $k\in\NN$ such that the associated set (also called \textbf{multiplicative tile})
 \[
  W = \bigcup_{i=1}^m\{ C^{-T} \xi_i : C \in \CC \}
 \] is translationally $k$-tiling with respect to $\Gamma$, then there exist $\Lambda \subset \mathbb{R}^d$ such that
 $W$ is a $(\mathcal{D},\Lambda)$-Riesz-wavelet set. 
\end{theorem}

\begin{proof}
To use Theorem \ref{thm:main}, we need to verify that $W$ is a measurable, Riemann integrable set that is multiplicatively tiling w.r.t. $\mathcal{D}^{-T}$.
Measurability and Riemann integrability are provided by the assumptions on $\CC$: We start out by observing that each $\Phi_i : H \to \mathcal{O}_i$ is a homeomorphism, with the additional property that the image of the Haar measure under $\Phi$ is equivalent to the Lebesgue measure \cite{BeTa}. In particular, the subset $\Phi_i(\CC) \subset \mathcal{O}_i$ is a Borel set with boundary of measure zero and compact closure in $\mathcal{O}_i$. Since the $\mathcal{O}_i$ are open and pairwise disjoint, this implies that the boundary of the union $W = \bigcup_{i=1}^m \Phi_i(\CC)$ is the union of the boundaries of the $\Phi_i(\CC)$, hence again of measure zero. Thus $W$ is Borel-measurable and Riemann-integrable. 

For the multiplicative tiling property, note that 
 \[
  \Phi_i (A_1A_2) = (A_1 A_2)^{-T} (\xi_i) = A_1^{-T} A_2^{-T} \xi_i = A_1^{-T} \Phi_i(A_2). 
 \]
This observation entails
\begin{align}\label{complementEquivalence}
 \sum_{A \in \mathcal{D}} \mathbb{1}_{A \CC} = 1 \mbox{ a.e. on $H$} \Leftrightarrow 
 \sum_{A \in \mathcal{D}} \mathbb{1}_{A^{-T} W} = 1 \mbox{ a.e. on } \mathbb{R}^d,
\end{align}
since $\mathbb{R}^d \setminus \bigcup_{i=1}^m \mathcal{O}_i$ is a set of measure zero. Because $\mathcal{D}$ is a quasi lattice with complement $\CC$ the assumptions of Theorem \ref{thm:cwt_basis} imply the left hand side of (\ref{complementEquivalence}). Then the right hand side holds as well which is equivalent to $W$ being multiplicatively tiling with respect to $\mathcal{D}^{-T}$.
\end{proof}

\begin{remark} Note that the set $W$ depends not only on $\mathcal{C}$, but also on the choice of the $\xi_i$. This choice does not affect the multiplicative tiling properties of $W$, but it can be expected to have a crucial impact on its {\em translational tiling} properties; that is why these properties are contained in the assumptions of the theorem. We expect that for the systematic application of Theorem \ref{thm:cwt_basis} for the construction of wavelet Riesz bases, the proper choice of the $\xi_i$ is the main obstacle. 
\end{remark}

\begin{remark} \label{rem:sampling} 
The fundamental relationship between continuous and discrete wavelet transforms is realized by observing that the discrete wavelet coefficients $\langle f, D_A T_\lambda \psi \rangle$ are just samples of the continuous transform:
\[
   \mathcal{W}_\psi f (A \lambda, A)=\langle f, D_A T_\lambda \psi \rangle.  
\]
Thus the continuously labelled wavelet system $(T_x D_A \psi)_{(x,A) \in G}$ is replaced by a discretely labelled subsystem, with very similar characteristics. In particular, the frame inequalities associated to the wavelet Riesz basis,  
\[
  C_1 \| f \|_2^2 \le \sum_{\lambda \in \Lambda, A \in \mathcal{D}} |W_\psi f(A\lambda,A)|^2 \le  C_2 \| f \|_2^2.
\]  can be combined with Equation (\ref{eqn:cwt_isometric}) to yield a {\em sampling theorem} for the image space of $\mathcal{W}_\psi$, namely 
\[
 \frac{C_1}{c_\psi} \| \mathcal{W}_{\psi} f \|_2^2 \le \sum_{\lambda \in \Lambda, A \in \mathcal{D}} |W_\psi f(A\lambda,A)|^2 \le  \frac{C_2}{c_\psi}  \| W_\psi f \|_2^2. 
\] I.e., each $\mathcal{W}_\psi f$ can be robustly reconstructed from its values on the {\em sampling grid} $\{(A\lambda,A): A \in \mathcal{D}, \lambda \in \Lambda \}$. 

Here, the Riesz basis property of the subsystem can be understood as {\em removing the redundancies} from the continuously labelled system: 
Using the dual basis via (\ref{eqn:interpol}), we see that for each square-summable coefficient family $(c_{\lambda,A})_{\lambda \in \Lambda, A \in \mathcal{D}}$ there exists a unique $f \in L^2(\mathbb{R}^d)$ satisfying 
\[ \forall (\lambda,A) \in \Lambda \times \mathcal{D}~:~ W_\psi f (A\lambda,A) = c_{\lambda,A}.\]  In the parlance of sampling theory, we have that the set \mbox{$\{ (A\lambda,A) : A \in \mathcal{D}, \lambda \in \Lambda \}$} is an {\em interpolating set of sampling} for the image space $\mathcal{W}_\psi(L^2(\mathbb{R}^d))$ of the wavelet transform. 
\end{remark}
\section{Examples}
\label{sec:examples}

We will now exhibit two instances in dimension $d=2$, where Theorem \ref{thm:cwt_basis} and Theorem \ref{thm:main}, respectively, can be applied.

\subsection*{Shearlet Riesz bases}
\label{sect:Shearlet}

Shearlets have been the subject of intensive research activity in recent years, and the construction of shearlet frames is a particularly active branch, albeit often with focus on {\em cone-adapted} shearlets, which do not quite fit our framework. As a small sample representing a larger body of work, we mention \cite{Gr,GuLa,KiKuLi}. 

The {\bf shearlet dilation group} $H$ is the group generated by scaling matrices $A_a,$~\mbox{$a\in \RR^+$}, and shearing matrices $S_m,~m\in\RR$, which are defined by 
\begin{align}
A_a=
\begin{pmatrix}
a & 0  \\
0 & a^{c}
\end{pmatrix}
\text{ and }
S_m=
\begin{pmatrix}
1 & m  \\
0 & 1
\end{pmatrix}.
\label{matrixdefinition}
\end{align}
Here $c \in \mathbb{R}$ can be viewed as an anisotropy parameter of the shearlet system, which will be assumed fixed throughout this subsection. 
Then
\[
 H = \{ S_m A_a : m \in \mathbb{R}, a \in \mathbb{R}^+ \},
\]  is easily seen to constitute a closed subgroup of $Gl(2,\mathbb{R})$, using the relations
\begin{equation} \label{eqn:relations_shearlet}
 A_{a_1} A_{a_2} = A_{a_1a_2}~,~ S_{m_1} S_{m_2} = S_{m_1 + m_2}~,~ A_a S_m = S_{a^{1-c} m} A_a. 
\end{equation} Furthermore, the factorization $h = S_m A_a\in H$ is in fact unique.

Finally, the computation 
\[
 (S_m A_a)^T \left( \begin{array}{c} \pm 1 \\ 0 \end{array} \right) =  \pm
 \left( \begin{array}{c}  a \\ m a^c \end{array} \right)
\] shows that $\pm \mathbb{R}_{>0} \times \mathbb{R}$ are open free dual orbits, with complement of measure zero.
We next exhibit a quasi-lattice in $H$, with associated  multiplicative tile in $\mathbb{R}^2$:
\begin{lemma}
\label{lemma:multTile}
Fix $a>1$, the scaling parameter, and $b>0$, the shearing parameter, and define
\[
 \mathcal{D}_{a,b} = \{ A_{a^k} S_{mb} : k, m \in \mathbb{Z} \}\subset H,
\] as well as 
\[
\CC_{a,b} = \left\{ S_y A_s : y \in \left (-\frac{b}{2} , \frac{b}{2} \right), s \in (a^{-1},1) \right\}\subset H. 
\] Then $\mathcal{D}_{a,b}$ is a quasi-lattice with complement $\CC_{a,b}$. The multiplicative tile constructed from $\CC$ and $\xi_{1/2} = \pm (1,0)^T$ is
\[
 W^{a,b} = \left\{ \left( \begin{array}{c} x \\ y \end{array} \right) : |x| \in (1,a), |y| < \frac{b|x|}{2} \right\}.
\]
\end{lemma}
\begin{proof}
 In order to prove the quasi-lattice property (\ref{def:quasiLattice}) for $\mathcal{D}$ with complement $\CC$, we need to show that, for almost every pair $(t,r) \in \mathbb{R}_{>0} \times \mathbb{R}$, the equation
 \[
  S_r A_t =  A_{a^k} S_{mb}  S_y A_s~,k,m \in \mathbb{Z},  y \in \left (-\frac{b}{2} , \frac{b}{2} \right), s \in (a^{-1},1)
 \] is uniquely solvable. 
 Using (\ref{eqn:relations_shearlet}), together with the unique factorization property, this is equivalent to uniquely solving
 \[
  t = a^k s ~, ~r = a^{(1-c)k} (mb+y)~.
 \] Now for $t \not\in a^\mathbb{Z}$, the first equation has a unique solution $k \in \mathbb{Z}, s \in (a^{-1},1)$. With $k$ fixed, one then sees that the second equation is uniquely solvable for $(m,y)$, whenever $ra^{-(1-c)k} \not\in b \mathbb{Z}$. Thus $\mathcal{D}_{a,b}$ is a quasi-lattice with complement $\CC_{a,b}$. 
 
 Following the description in Theorem \ref{thm:cwt_basis}, with $\xi_{1/2} = \pm (1,0)^T$, we get
 \begin{eqnarray*}
  W^{a,b} & = & \left\{ \pm (S_y A_s)^{-T} \left( \begin{array}{c} 1 \\ 0 \end{array} \right)~: y \in \left (-\frac{b}{2} , \frac{b}{2} \right), s \in (a^{-1},1) \right\} \\
  & = &  \left\{ \pm \left( \begin{array}{c} -s^{-1} \\ -s^{-1} y \end{array} \right)~: y \in \left (-\frac{b}{2} , \frac{b}{2} \right), s \in (a^{-1},1) \right\} \\
  & = & \left\{ \left( \begin{array}{c} x \\ y \end{array} \right) : |x| \in (1,a), |y| < \frac{b|x|}{2} \right\}.
 \end{eqnarray*}
\end{proof}
The set $W=W^{a,b}$ is visualized by the grey shaded area in Figure \ref{WZerlegung}, which also gives a graphical  indication why the action of $\mathcal{D}^{-T}$ provides a multiplicative tiling. In order to prove the existence of a shearlet Riesz basis, it remains to verify the translational $k$-tiling property of $W^{a,b}$ from Lemma \ref{lemma:multTile}. 
\begin{lemma}
\label{mykTilingTheorem}
Let $a>1,b>0$ be rational. Then there exists a lattice $\Gamma \subset \mathbb{R}^2$ and $k \in \mathbb{N}$ such that
$W^{a,b}$ is $k$-tiling with respect to $\Gamma$.
\end{lemma}

\begin{proof}
Shifting a set by an element of a lattice does not affect its tiling property with regard to the same lattice. 
To show the translational $k$-tiling property we partition $W=W^{a,b}$ into subsets and shift and merge the subsets to obtain a set $W'$ whose $k$-tiling property provides that $W$ itself is $k$-tiling.

 The set $W$ and the following partition into subsets are indicated in Figure \ref{WZerlegung}.
\begin{align*}
W_+ & =\left\{(x,y)\in W \mid x>0\right\}, &
W_{+,1} & =\left\{(x,y)\in W_+ \left| |y|\leq \frac{b}{2}\right.\right\},\\
W_{+,2} & =\left\{(x,y)\in W_+ \left| \frac{b }{2}< y\right.\right\}, &
W_{+,3}& =\left\{(x,y)\in W_+ \left| y< -\frac{b}{2}\right.\right\}.
\end{align*}
Note that $W$ is point symmetric with respect to $(0,0)$. We also define $W_{-}=-W_+$ and $W_{-,i} =-W_{+,i},~(i=1,2,3)$. We obtain $W=W_+\dcup W_-$, $W_+=W_{+,1}\dcup W_{+,2}\dcup W_{+,3}$ and $W_{-}=W_{-,1}\dcup W_{-,2}\dcup W_{-,3}$.
We shift the sets $W_{\pm,i}$ by certain vectors to and merge the results to obtain a rectangular set $W'$, whose $k$-tiling property is easy to determine. We will then have to make sure that there exists a lattice containing the translations employed in the process, for which $W'$ is $k$-tiling; here the rationality conditions on $a$ and $b$ come into play. 

Now let us describe the shift-and-merge procedure in more detail. 
In the following computation, note that all set theoretic equalities in this proof have to be understood up to measure zero sets, and all unions are understood as disjoint.
\begin{align*}
W' &=\left(W_{+,1}+\left(\begin{array}{c} -1  \\ 0 \end{array}\right)\right)  \cup \left(W_{-,1}+\left(\begin{array}{c} 1 \\ 0 \end{array}\right)\right)\\
&~\cup \left(W_{+,2}+\left(\begin{array}{c} -1  \\ 0 \end{array}\right)\right)\cup \left(W_{-,2}+\left(\begin{array}{c} 2 \\ \frac{ab+b}{2} \end{array}\right)\right)\\
&~\cup \left(W_{+,3}+\left(\begin{array}{c} -2 \\ \frac{ab+b}{2} \end{array}\right)\right)\cup \left(W_{-,3}+\left(\begin{array}{c} 1  \\ 0 \end{array}\right)\right)\\
& =[-(a-1),a-1]\times \left[-\frac{b}{2},\frac{ab+b}{2}\right].
\end{align*}
Since we assumed $a$ to be rational, there exists $\alpha \in \mathbb{Q}$ such that $\{1,2a-2\} \subset \alpha \mathbb{Z}$.  Similarly, there exists $\beta \in \mathbb{Q}$ such that $\{ \frac{ab+b}{2}, \frac{ab+2b}{2} \}
\in \beta \mathbb{Z}$. But this implies that all translations employed in the construction of $W'$ are contained in $ \Gamma = \alpha \mathbb{Z} \times \beta \mathbb{Z}$, and in addition, $W'$ is $k$-tiling with respect to $\Gamma$, for 
\[
 k = \frac{2a-2}{\alpha} \cdot \frac{ab+2b}{2\beta} \in \mathbb{N}.
\]
Thus $W$ is $k$-tiling. 
\end{proof}

With Lemma \ref{lemma:multTile}, Lemma \ref{mykTilingTheorem} and Theorem \ref{thm:cwt_basis} we deduce the following result: 
\begin{corollary}
 For any choice of rational numbers $a>1$ and $b>0$, there exists a $\psi \in {\rm L}^2(\mathbb{R}^2)$ and a set $\Lambda \subset \mathbb{R}^2$ such that $\left( D_{A_{a^k}} D_{S_{mb}} T_\lambda \psi \right)_{k,m \in \mathbb{Z}, \lambda \in \Lambda}$ is a Riesz basis. 
\end{corollary}

\begin{remark} 
Note that the discretization of both the shearing and the scaling parameter  in Lemma \ref{lemma:multTile} and its corollary can be chosen arbitrarily fine. 

Similar discretizations of $\mathbb{R}^d \rtimes H$  have been considered before, (see \cite{DaKuStTe, DAHLKE10}), in particular the multiplicative tiling induced by the shearlet group is well-known. 
The main difference between our construction and the discretizations in \cite{DaKuStTe, DAHLKE10} lies in the set $\Lambda$, which determines the translational part of the discretization. In  \cite{DaKuStTe, DAHLKE10} the set $\Lambda$ is a lattice whereas we employ a quasicrystal. 

Furthermore, note that the results in this section hold independently of the choice of the anisotropy parameter $c$. In the shearlet literature it is customary to consider only $c=1/2$. Our proof shows that the {\em same} set $W$ serves as  Riesz wavelet set for a whole class of distinct dilation groups. 
\end{remark}
\begin{remark}
The construction of cone-adapted shearlets can be roughly described as a  two-step procedure: First the frequencies are decomposed into two disjoint cones, and then each cone is subjected to a shearlet-type tiling. For suitable choices of $a$ and $b$, our construction of shearlet tilings is fully compatible with this decomposition, hence we expect that there exist cone-adapted shearlet Riesz bases, as well.  

There are however variations, for which it is currently open whether a Riesz wavelet exists. For instance, extending the group $H$ by including the simple reflection $-I_2$ apparently makes the problem of finding such a set much harder: Our proof of Lemma \ref{mykTilingTheorem} crucially depends on the fact that $W^{a,b}$ is the union of {\em two} wedge-shaped sets, which can be cut up and recombined to yield a rectangle. 

Furthermore, the generalization to dimension $n=3$ is also not clear to us. 
\end{remark}

\begin{figure}
\centering

\begin{tikzpicture}[
    scale=5,
    axis/.style={very thick, ->, >=stealth'},
    important line/.style={thick},
    dashed line/.style={dashed, thin},
    pile/.style={thick, ->, >=stealth', shorten <=2pt, shorten
    >=2pt},
    every node/.style={color=black}
    ]
    \draw[axis] (0cm,0.5)  -- (2.0,0.5) node(xline)[right]
        {$\RR$};
    \draw[axis] (1.0,-0.9cm) -- (1.0,1.9) node(yline)[above] {$\RR$};

		\draw[line width=0.35mm] (1.45,0.28cm) -- (1.45,0.72);
		\draw[line width=0.35mm] (1.9,0.06cm) -- (1.9,0.95);
		
		\draw[line width=0.35mm] (1.45,0.72cm) -- (1.9,0.95);
		\draw[line width=0.35mm] (1.45,0.28cm) -- (1.9,0.06);
		
		\draw[line width=0.35mm] (1.45,0.72cm) -- (1.9,0.72);
		\draw[line width=0.35mm] (1.45,0.28cm) -- (1.9,0.28);
	
		\draw[line width=0.35mm] (0.55,0.28cm) -- (0.55,0.72);
		\draw[line width=0.35mm] (0.1,0.06cm) -- (0.1,0.95);
		
		\draw[line width=0.35mm] (.55,0.72cm) -- (0.1,0.95);
		\draw[line width=0.35mm] (.55,0.28cm) -- (0.1,0.06);

		\draw[line width=0.35mm] (0.55,0.72cm) -- (0.1,0.72);
		\draw[line width=0.35mm] (0.55,0.28cm) -- (0.1,0.28);

				\draw[line width=0.35mm] (0.97,0.72cm) -- (1.0,0.72);
				\draw[line width=0.35mm] (0.97,0.28cm) -- (1.0,0.28);
				\draw[line width=0.35mm] (0.97,0.95cm) -- (1.0,0.95);
				\draw[line width=0.35mm] (0.97,0.06cm) -- (1.0,0.06);
				\draw[] (0.9,0.72) node {$\frac{b}{2}$};
				\draw[] (0.9,0.28) node {$-\frac{b}{2}$};
				\draw[] (0.9,0.95) node {$\frac{ab}{2}$};
				\draw[] (0.9,0.06) node {$-\frac{ab}{2}$};

				\draw[line width=0.35mm] (1.9,0.95cm) -- (1.9,1.84);
				\draw[line width=0.35mm] (1.45,0.72cm) -- (1.45,1.16);
				\draw[line width=0.35mm] (1.45,1.16cm) -- (1.9,1.84);

				\draw[line width=0.35mm] (0.1,0.95cm) -- (0.1,1.84);
				\draw[line width=0.35mm] (0.55,0.72cm) -- (0.55,1.16);
				\draw[line width=0.35mm] (0.55,1.16cm) -- (0.1,1.84);

	\draw[line width=0.35mm] (0.55,0.28cm) -- (0.55,-0.16);
		\draw[line width=0.35mm] (0.1,0.06cm) -- (0.1,-0.83);
	\draw[line width=0.35mm] (0.55,-0.16) -- (0.1,-0.83);
	
	\draw[line width=0.35mm] (1.45,0.28cm) -- (1.45,-0.16);
		\draw[line width=0.35mm] (1.9,0.06cm) -- (1.9,-0.83);
	\draw[line width=0.35mm] (1.45,-0.16) -- (1.9,-0.83);
	
	\begin{scope} 
\clip (1.9,0.06) -- (1.9,0.95) --(1.45,0.72) --(1.45,0.28);
\fill[gray,opacity=.4] (0,0) rectangle (2,2); 
\end{scope}

	\begin{scope} 
\clip (0.1,0.06) -- (0.1,0.95) --(0.55,0.72) --(0.55,0.28);
\fill[gray,opacity=.4] (0,0) rectangle (2,2); 
\end{scope}
	
\draw [decorate,decoration={brace,amplitude=8pt},xshift=0pt,yshift=0pt]
(1.9+0.2,0.72) -- (1.9+0.2,0.28)node [black,midway,xshift=30pt] {$W_{+,1}$};

\draw [decorate,decoration={brace,amplitude=8pt},xshift=0pt,yshift=0pt]
(1.9+0.2,0.95) -- (1.9+0.2,0.72)node [black,midway,xshift=30pt] {$W_{+,2}$};

\draw [decorate,decoration={brace,amplitude=8pt},xshift=0pt,yshift=0pt]
(1.9+0.2,0.28) -- (1.9+0.2,0.06)node [black,midway,xshift=30pt] {$W_{+,3}$};

\draw [decorate,decoration={brace,amplitude=8pt,mirror},xshift=0pt,yshift=0pt]
(0.1,0.72) -- (0.1,0.28)node [black,midway,xshift=-30pt] {$W_{-,1}$};

\draw [decorate,decoration={brace,amplitude=8pt,mirror},xshift=0pt,yshift=0pt]
(0.1,0.95) -- (0.1,0.72)node [black,midway,xshift=-30pt] {$W_{-,3}$};

\draw [decorate,decoration={brace,amplitude=8pt,mirror},xshift=0pt,yshift=0pt]
(0.1,0.28) -- (0.1,0.06)node [black,midway,xshift=-30pt] {$W_{-,2}$};

\draw [decorate,decoration={brace,amplitude=8pt,mirror},xshift=0pt,yshift=0pt]
(1.90+0.2,0.95) -- (1.9+0.2,1.84)node [black,midway,xshift=30pt] {$S_1^{-T} W_+$};

\draw [decorate,decoration={brace,amplitude=8pt,mirror},xshift=0pt,yshift=0pt]
(1.90+0.2,-0.83) -- (1.9+0.2,0.06)node [black,midway,xshift=30pt] {$S_{-1}^{-T} W_+$};

\draw [decorate,decoration={brace,amplitude=8pt},xshift=0pt,yshift=0pt]
(0.1,-0.83) -- (0.1,0.06)node [black,midway,xshift=-30pt] {$S_1^{-T}  W_-$};

\draw [decorate,decoration={brace,amplitude=8pt},xshift=0pt,yshift=0pt]
(0.1,0.95) -- (0.1,1.84)node [black,midway,xshift=-30pt] {$S_{-1}^{-T}  W_-$};
\end{tikzpicture}
\caption{The grey shaded area describes the set $W$ and its decomposition used in Lemma \ref{mykTilingTheorem}. To illustrate the multiplicative tiling property, we show $S_1^{-T}W$ and $S_{-1}^{-T}W$.}
\label{WZerlegung}
\end{figure}
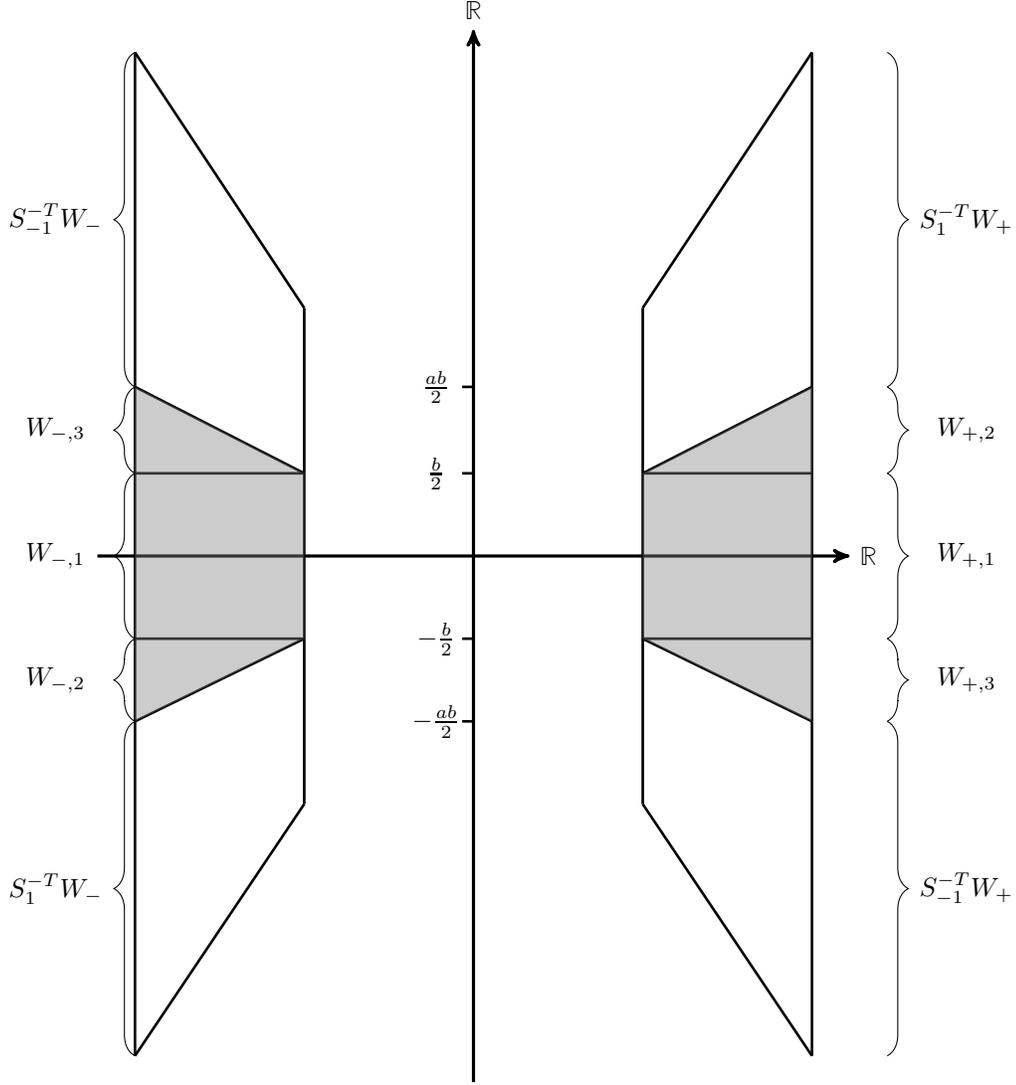

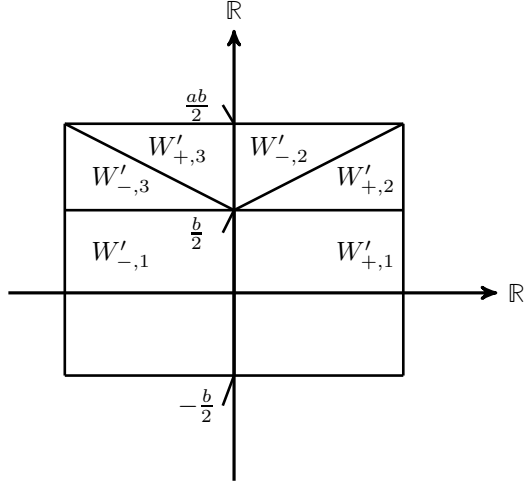
\begin{figure}
\centering
\begin{tikzpicture}[
    scale=5,
    axis/.style={very thick, ->, >=stealth'},
    important line/.style={thick},
    dashed line/.style={dashed, thin},
    pile/.style={thick, ->, >=stealth', shorten <=2pt, shorten
    >=2pt},
    every node/.style={color=black}
    ]
    \draw[axis] (0.4cm,0.5)  -- (1.7,0.5) node(xline)[right]
        {$\RR$};
    \draw[axis] (1.0,0cm) -- (1.0,1.2) node(yline)[above] {$\RR$};

		\draw[line width=0.35mm] (1.00,0.28cm) -- (1.00,0.72);
		\draw[line width=0.35mm] (1.45,0.28cm) -- (1.45,0.95);
		
		\draw[line width=0.35mm] (1.00,0.72cm) -- (1.45,0.95);
		\draw[line width=0.35mm] (1.00,0.95cm) -- (1.45,0.95);
		
		\draw[line width=0.35mm] (1.00,0.72cm) -- (1.45,0.72);
		\draw[line width=0.35mm] (1.00,0.28cm) -- (1.45,0.28);
		\draw[] (1.35,0.8) node {$W'_{+,2}$};
			\draw[] (1.35,0.6) node {$W'_{+,1}$};
				\draw[] (0.85,0.88) node {$W'_{+,3}$};
	
		\draw[line width=0.35mm] (1.00,0.28cm) -- (1.00,0.72);
		\draw[line width=0.35mm] (0.55,0.28cm) -- (0.55,0.95);
		
		\draw[line width=0.35mm] (1.00,0.72cm) -- (0.55,0.95);
		\draw[line width=0.35mm] (1.00,0.95cm) -- (0.55,0.95);

		\draw[line width=0.35mm] (1.00,0.72cm) -- (0.55,0.72);
		\draw[line width=0.35mm] (1.00,0.28cm) -- (0.55,0.28);
		\draw[] (0.7,0.8) node {$W'_{-,3}$};
			\draw[] (0.7,0.6) node {$W'_{-,1}$};
				\draw[] (1.12,0.88) node {$W'_{-,2}$};

				\draw[line width=0.35mm] (0.97,0.66cm) -- (1.0,0.72);
				\draw[line width=0.35mm] (0.97,0.20cm) -- (1.0,0.28);
				\draw[line width=0.35mm] (0.97,1.00cm) -- (1.0,0.95);
				\draw[] (0.9,0.66) node {$\frac{b}{2}$};
				\draw[] (0.9,0.20) node {$-\frac{b}{2}$};
				\draw[] (0.9,1.00) node {$\frac{ab}{2}$};

\end{tikzpicture}
\caption{The set $W'$ which is obtained by shifting and merging subsets of $W$.}
\label{WPrimeZerlegung}
\end{figure}

\subsection*[Example 2: Isotropic Dilations and Rotations]{Example 2: Isotropic Dilations and Rotations} 
\label{section:rotation}
In this example we pursue the same strategy as in the construction of a Shearlet basis, but this time the associated dilation group is the {\bf similitude group}, i.e. the group 
\[
 H  = \left\{ r \left( \begin{array}{cc} \cos(\theta) & -\sin(\theta) \\ \sin(\theta) & \cos(\theta) \end{array} \right): r>0, \theta \in \mathbb{R} \right\}  
\] consisting of rotations and scalar dilations of the plane. This group was introduced to wavelet analysis by \cite{Mu}. The aim of this subsection is to show the existence of Riesz wavelet sets adapted to arbitrarily fine resolutions of the rotation parameter. 

For even $n\in\NN_{>2}$ we consider the rotation matrix $R_n$ given by
\begin{align*}
R_n=
\begin{pmatrix}
\cos{\frac{2\pi}{n}} & -\sin{\frac{2\pi}{n}} \\ 
\sin{\frac{2\pi}{n}} & \cos{\frac{2\pi}{n}}
\end{pmatrix}
\end{align*}
and define the set
  $\mathcal{D}=\{a^k(R_n)^l\mid k\in\ZZ, l=0,\ldots,\frac{n}{2}-1\}\subset H$.
$R_n$ rotates the plane around $(0,0)$ with rotation angle $\frac{2\pi}{n}$; thus $(R_n)^l$ rotates with rotation angle $l\cdot \frac{2\pi}{n}$.

Instead of going via Theorem \ref{thm:cwt_basis}, we have found it more convenient to use Theorem \ref{thm:main} directly, that is, to choose a $V$, and show that it is $k$-tiling and $\mathcal{D}^{-T} \cdot V$ is a partition of $\RR^2$.
\begin{defi}[V]
For $a\in\NN_{\geq 2}$ and even $n\in\NN_{>2}$  define
\begin{align}
V=V^{a,n} & =\left\{(x,y)\in\RR^2 \left| |y|\in (1,a],~ |x|\leq \tan\left(\frac{\pi}{n}\right)|y| \right.\right\}.
\end{align}
\end{defi}
Since $a$ and $n$ are fixed throughout the example we usually omit them as superscripts, in order to simplify notation.
The formal definition of $V$ looks somewhat cryptic, and is best understood geometrically: $V$ can be constructed with the $n$ regular polygon centered around $(0,0)$ as illustrated in \mbox{Figure \ref{mySetV}}.
Let $K_n$ denote $n$ regular polygon with vertices on the boundary of $B_{r}(0)$, $r=(cos(\frac{\pi}{n}))^{-1}\delta a$, and let its rotation be such that one vertex has the coordinates $\left(\tan(\frac{\pi}{2})\delta a,\delta a\right)$. 
Then let $P_n$ equal $K_n\backslash a^{-1}K_n$, i.e., we cut out a smaller version of the polygon as indicated by \mbox{Figure \ref{mySetV}}. $P_n$ and isotropic scalings of $P_n$ are called $n$-regular polygonal rings. Then $P_n$ can be partitioned into $n$ congruent areas and $V$ equals to two opposing areas of $P_n$ (the proof is analogous for all choices of pairs of opposing areas).

The shape of $V$ is very similar to the shape of $W$ in Example 1, but rotated by an angle of $\frac{\pi}{2}$. 
However the shapes of $V_+$ and $W_+$ respectively differ as the slopes of the side elements are different.
Nevertheless, the $k$-tiling proof is almost exactly the same as in Lemma \ref{mykTilingTheorem}:
\begin{lemma}
\label{myKtilingRotation}
For $a\in\NN_{\geq 2}$ and an even $n\in\NN_{>2}$ the set
$V^{a,n}$ is $k$-tiling for some $k\in\NN$.
\end{lemma}
\begin{proof}

Similar to the proof of Lemma \ref{mykTilingTheorem} we decompose $V$ into subsets, move them by elements of a lattice $\Gamma$ and merge them to obtain a set $V'$. Then the $k$-tiling property of $V'$ with lattice $\Gamma$ is easier to determine and leads to the $k$-tiling property of $V$. Let $b=\tan\left(\frac{\pi}{n}\right)\delta$. The following partition of $V$ into subsets is illustrated in Figure \ref{mySetV}. 
\begin{align*}
V_+ & =\left\{(x,y)\in V \left| y>0\right.\right\}, & 
V_{+,1} & =\left\{(x,y)\in V_+ \left| |x|\leq b\right.\right\},\\
V_{+,2} & =\left\{(x,y)\in V_+ \left| x<-b\right.\right\}, & 
V_{+,3}& =\left\{(x,y)\in V_+ \left| b<x\right.\right\}.
\end{align*}
Note that $V$ is point symmetric with respect to $(0,0)$ and define $V_{-}=-V_+$, \mbox{$V_{-,i}  =-V_{+,i},$} $(i=1,2,3)$.
We obtain $V=V_+\dcup V_-$, $V_+=V_{+,1}\dcup V_{+,2}\dcup V_{+,3}$ and $V_{-}=V_{-,1}\dcup V_{-,2}\dcup V_{-,3}$.
Fix the lattice $\Gamma=b(a+1)\ZZ\times \delta\ZZ$ and obtain:
\begin{align*}
V' &=\left(V_{+,1}+\left(\begin{array}{c} 0  \\ -\delta \end{array}\right)\right)  \cup \left(V_{-,1}+\left(\begin{array}{c} 0  \\ \delta \end{array}\right)\right)\\
&~ \cup \left(V_{+,2}+\left(\begin{array}{c} 0  \\ -\delta \end{array}\right)\right)\cup \left(V_{-,2}+\left(\begin{array}{c} -b(a+1) \\ \delta a \end{array}\right)\right)\\
&~ \cup \left(V_{+,3}+\left(\begin{array}{c}-b(a+1) \\ -\delta a \end{array}\right)\right)\cup \left(V_{-,3}+\left(\begin{array}{c}0  \\ \delta \end{array}\right)\right)\\
& =[-ba,b]\times[-\delta (a-1),\delta(a-1)].
\end{align*}
Now $F=(0,b(a+1)]\times(0,\delta]$ is a fundamental domain for $\Gamma$ and $V'$ consists of $2(a-1)$ shifted copies of $F$. Thus $V'$ is $k$-tiling with $k=2(a-1)$. Because all translations used in the construction of $V'$ are in $\Gamma$ the set $V$ is $k$-tiling w.r.t. $\Gamma$.
\end{proof}
\begin{lemma}
\label{myRotation}$V$ tiles $\RR^2$ under the action of $\mathcal{D}^{-T}$, that is
\[\mathcal{D}^{-T} \cdot V=\{g^{-T}V \mid g\in \mathcal{D}\}\]
is a partition of $\RR^2$.
\end{lemma}
\begin{proof}
For $k\in \ZZ$ and $l \in \mathbb{Z}$ define $g_{k,l}=a^k(R_n)^l$, then 
\[\mathcal{D}^{-T} = \{ g_{k,-l} : k \in \mathbb{N}, l = 0,\ldots, \frac{n}{2}-1 \}.\]

Fix $k\in\ZZ$ and consider the union over all rotations which results in a scaled version of the $n$-regular polygon ring, i.e., 

$\bigcup_{l\in \{0,\ldots, \frac{n}{2}-1\}}g_{k,-l}V=a^kP_n$.
Now, the union over $k\in\ZZ$ is the union of scaled $n$-regular polygon rings, which are disjoint in measure, and this union equals $\RR^2\backslash\{0\}$, more precisely:
\[\bigcup_{k\in\ZZ}\bigcup\limits_{l=0}^{\frac{n}{2}-1}g_{k,l}V=\bigcup_{k\in\ZZ} a^k P_n=\RR^2\backslash\{0\}.\]
\end{proof}

Thus, by appealing to \ref{thm:main}, we obtain a suitable $\Lambda \subset \mathbb{R}^2$ such that $V$ is a $(\Lambda,\mathcal{D})$-Riesz-wavelet set.

\begin{figure}
\centering
\begin{tikzpicture}[
    scale=4.0,
    axis/.style={very thick, ->, >=stealth'},
    important line/.style={thick},
    dashed line/.style={dashed, thin},
    pile/.style={thick, ->, >=stealth', shorten <=2pt, shorten
    >=2pt},
    every node/.style={color=black}
    ]
    \draw[axis] (0cm,0.5)  -- (2.2,0.5) node(xline)[right]
        {$\RR$};
    \draw[axis] (1.0,-0.33cm) -- (1.0,1.43) node(yline)[above] {};
		\draw (1,0.5) \polygon{6}{0.8};
		\draw (1,0.5) \polygon{6}{0.4};
		\def\myval{2*tan(360/12)*0.8}
		\def\myvali{2*tan(360/12)*0.4}
		\coordinate (x) at ({sin(360/12)*\myval+1},{cos(360/(12))*\myval+0.5});
		\coordinate (y) at ({sin(360/12)*\myvali+1},{cos(360/(12))*\myvali+0.5});
		\coordinate (x5) at ({sin(5*360/12)*\myval+1},{cos(5*360/(12))*\myval+0.5});
		\coordinate (y5) at ({sin(5*360/12)*\myvali+1},{cos(5*360/(12))*\myvali+0.5});
		\coordinate (x7) at ({sin(7*360/12)*\myval+1},{cos(7*360/(12))*\myval+0.5});
		\coordinate (y7) at ({sin(7*360/12)*\myvali+1},{cos(7*360/(12))*\myvali+0.5});
		\coordinate (x11) at ({sin(11*360/12)*\myval+1},{cos(11*360/(12))*\myval+0.5});
		\coordinate (y11) at ({sin(11*360/12)*\myvali+1},{cos(11*360/(12))*\myvali+0.5});

	  \coordinate (ytief) at ({sin(360/12)*\myvali+1},{cos(360/(12))*\myvali+0.5-0.05});
		\coordinate (yu) at ({sin(360/12)*\myvali+1},{cos(360/(12))*\myvali+0.5-0.08});
		\coordinate (yo) at ({sin(360/12)*\myvali+1},{cos(360/(12))*\myvali+0.5-0.02});
		
		\coordinate (y11tief) at ({sin(11*360/12)*\myvali+1},{cos(11*360/(12))*\myvali+0.5-0.05});
		\coordinate (y11u) at ({sin(11*360/12)*\myvali+1},{cos(11*360/(12))*\myvali+0.5-0.08});
		\coordinate (y11o) at ({sin(11*360/12)*\myvali+1},{cos(11*360/(12))*\myvali+0.5-0.02});
		
		 \coordinate (xtief) at ({sin(360/12)*\myval+1},{cos(360/(12))*\myval+0.5+0.05});
		\coordinate (xu) at ({sin(360/12)*\myval+1},{cos(360/(12))*\myval+0.5+0.08});
		\coordinate (xo) at ({sin(360/12)*\myval+1},{cos(360/(12))*\myval+0.5+0.02});
		
		\coordinate (x11tief) at ({sin(11*360/12)*\myval+1},{cos(11*360/(12))*\myval+0.5+0.05});
		\coordinate (x11u) at ({sin(11*360/12)*\myval+1},{cos(11*360/(12))*\myval+0.5+0.08});
		\coordinate (x11o) at ({sin(11*360/12)*\myval+1},{cos(11*360/(12))*\myval+0.5+0.02});
		
		
		\foreach \laufvari in {-1,1,5,7} 
		{
		\coordinate (x2) at ({sin(\laufvari*360/12)*\myval+1},{cos(\laufvari*360/(12))*\myval+0.5});
		\coordinate (y2) at ({sin(\laufvari*360/12)*\myvali+1},{cos(\laufvari*360/(12))*\myvali+0.5});
		\draw[-] (x2) --(y2);
	}	
	
	\coordinate (obeny11) at ({sin(11*360/12)*\myvali+1},{0.8+0.5});
	\draw[-] (obeny11) -- (y11);
	\coordinate (obeny) at ({sin(1*360/12)*\myvali+1},{0.8+0.5});
	\draw[-] (obeny) -- (y);
	\coordinate (unteny7) at ({sin(7*360/12)*\myvali+1},{-0.8+0.5});
	\draw[-] (unteny7) -- (y7);
	\coordinate (unteny5) at ({sin(5*360/12)*\myvali+1},{-0.8+0.5});
	\draw[-] (unteny5) -- (y5);
	

	
	\draw[-] (ytief) --(y11tief);
	\draw[-] (y11u) --  (y11o);
	\draw[-] (yu) --  (yo);
	\draw[-] (xtief) --(x11tief);
	\draw[-] (x11u) --  (x11o);
	\draw[-] (xu) --  (xo);
	
	\draw[fill=red!45,circle, inner sep=1pt] (1.12,0.80) node {$2b$};
	\draw[] (1.12,1.4) node {$2ba$};
	
	\begin{scope}
\clip (x) -- (y) --(y11) --(x11);
\fill[gray,opacity=.4] (0,0) rectangle (2,2); 
\end{scope}

\begin{scope}
\clip (y7) -- (y5) --(x5) --(x7);
\fill[gray,opacity=.4] (0,-1) rectangle (2,2); 
\end{scope}


\coordinate (obeny11hoch) at ({sin(11*360/12)*\myvali+1},{0.8+0.5+0.2});	
\coordinate (obenyhoch) at ({sin(1*360/12)*\myvali+1},{0.8+0.5+0.2});
\coordinate (x11hoch) at ({sin(11*360/12)*\myval+1},{cos(11*360/(12))*\myval+0.5+0.2});
\coordinate (xhoch) at ({sin(360/12)*\myval+1},{cos(360/(12))*\myval+0.5+0.2});
\draw [decorate,decoration={brace,amplitude=7pt},xshift=0pt,yshift=0pt]
(obeny11hoch) -- (obenyhoch)node [black,midway,yshift=16pt] {$V_{+,1}$};

\draw [decorate,decoration={brace,amplitude=7pt},xshift=0pt,yshift=0pt]
(x11hoch) -- (obeny11hoch)node [black,midway,yshift=16pt] {$V_{+,2}$};
\draw [decorate,decoration={brace,amplitude=7pt},xshift=0pt,yshift=0pt]
(obenyhoch) -- (xhoch)node [black,midway,yshift=16pt] {$V_{+,3}$};

\draw [decorate,decoration={brace,amplitude=7pt,mirror},xshift=0pt,yshift=0pt]
(unteny7) -- (unteny5)node [black,midway,yshift=-16pt] {$V{-,1}$};	
\draw [decorate,decoration={brace,amplitude=7pt,mirror},xshift=0pt,yshift=0pt]
(unteny5) -- (x5)node [black,midway,yshift=-16pt] {$V_{-,2}$};	
\draw [decorate,decoration={brace,amplitude=7pt,mirror},xshift=0pt,yshift=0pt]
(x7) -- (unteny7)node [black,midway,yshift=-16pt] {$V_{-,3}$};	
\end{tikzpicture}
\caption{The grey shaded area describes the set $V$ and its decomposition used in Lemma \ref{myKtilingRotation}.}
\label{mySetV}
\end{figure}
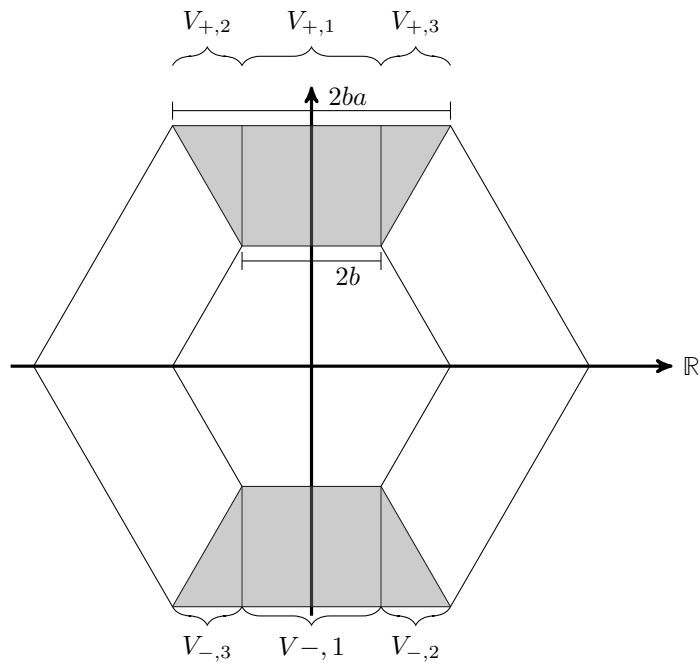

\section*{Concluding remarks}

While our exposition concentrated on the case of Riesz bases discretizing a continuous wavelet transform related to the choice of a  dilation group $H$, it seems obvious that the ideas can also be adapted to more general settings, such as composite dilation wavelets \cite{GuLa_etal}. 

A different, probably much more challenging, direction for possible generalization concerns the construction of Riesz bases with additional smoothness and vanishing moment properties. 

\bibliographystyle{amsplain}

\bibliography{referenzen}



%



\end{document}